\documentclass[12pt]{amsart}
\usepackage{amsmath, amsthm, amssymb}
\title{Dependence and Isolated Extensions}
\author{Vincent Guingona \\ Department of Mathematics \\ University of Maryland }
\thanks{Special thanks to Chris Laskowski.}
\date{\today}

\newtheorem{thm}{Theorem}[section]
\newtheorem{cor}[thm]{Corollary}
\newtheorem{lem}[thm]{Lemma}
\newtheorem{prop}[thm]{Proposition}

\theoremstyle{remark}
\newtheorem{rem}[thm]{Remark}

\theoremstyle{definition}
\newtheorem{defn}[thm]{Definition}

\newcommand{\dom}{\mathrm{dom} }
\newcommand{\tp}{\mathrm{tp} }
\newcommand{\concat}{{}^{\frown} }
\newcommand{\ID}{\mathrm{ID} }
\newcommand{\Avtp}{\mathrm{Av} }

\begin{document}

\maketitle

\begin{abstract}
 In this paper, we show that if $\varphi(x; y)$ is a dependent formula, then all $\varphi$-types $p$ have an extension to a $\varphi$-isolated $\varphi$-type, $p'$.  Moreover, we can choose $p'$ to be a \textit{elementary $\varphi$-extension} of $p$ (see Definition \ref{elmexten} below) and so that $|\dom(p') - \dom(p)| \le 2 \cdot \ID(\varphi)$.  We show that this characterizes $\varphi$ being dependent.  Finally, we give some corollaries of this theorem and draw some parallels to the stable setting.
\end{abstract}

\section{Introduction}\label{Section_intro}

There is a characterization of the stability of a formula $\varphi(\overline{x}; \overline{y})$ in terms of the definability of all $\varphi$-types.  A partitioned formula $\varphi(\overline{x}; \overline{y})$ is stable if and only if all $\varphi$-types are definable by a formula over their domain  \cite{shelah}.  We create an analogous result for dependent formulas (that is, formulas without the independence property, sometimes referred to as ``NIP'' formulas).  Since dependence is a strictly weaker notion than stability, we cannot hope to have definability of $\varphi$-types over their domain for general dependent formulas, $\varphi$.  However, we change the conclusion slightly, in two separate ways, and get a characterization of dependent formulas.

First, we weaken the requirement that a $\varphi$-type $p$ be definable over $\dom(p)$.  Instead, we take a model $M$ containing $\dom(p)$, take an elementary extension $(N; B)$ of the pair structure $(M; \dom(p))$, and demand that $p$ be definable over $B$.  Second, we strengthen the method by which the $\varphi$-type $p$ is definable.  Instead of being merely definable over this expanded set $B$, we demand that there exists an extension of $p$ to a $\varphi$-type $p'$ such that $\dom(p') \subseteq B$ and $p'$ is $\varphi$-isolated.  From all of this, we construct an analogous result to the characterization of stable formulas, the Isolated Extension Theorem (Theorem \ref{elmextiso} below).  The proof of this theorem is loosely based on a paper by Shelah \cite{shelah900}.

In Section \ref{Section_defsandmainthm} we discuss definitions, state the main theorem, and list some consequences of that theorem.  The main theorem, Theorem \ref{elmextiso}, is proved in Section \ref{Section_proofofmainthm}.  Finally, in the Section \ref{Section_stablesetting}, we discuss the implications of this theorem to the stable case.  Even in the stable case, Theorem \ref{elmextiso} provides new information.

\section{Definitions and The Isolated Extension Theorem}\label{Section_defsandmainthm}

Fix a complete, first-order theory $T$ in a language $L$.  We include the case where $L$ is multi-sorted, so we need to keep track of the sorts of variables.  For convenience, if $\psi(\overline{x})$ is any formula, then let $\psi(\overline{x})^0 = \neg \psi(\overline{x})$ and let $\psi(\overline{x})^1 = \psi(\overline{x})$.

For the first three definitions, fix $\varphi(\overline{x}; \overline{y})$ a partitioned formula of $L$.  By a $\varphi$-type, we mean a consistent set of formulas $p(\overline{x}) = \{ \varphi(\overline{x}; \overline{b})^{s(\overline{b})} : \overline{b} \in B \}$ for some set $B$ of elements of the same sort as $\overline{y}$ and some $s \in {}^B 2$ (the set of functions from $B$ to $2 = \{0, 1\}$).  We say that $\dom(p) = B$ and the space of all $\varphi$-types over $B$ is denoted

\begin{equation}
 S_\varphi(B) = \{ p(x) \mathrm{\ a\ } \varphi \mathrm{-type} : \dom(p) = B \}
\end{equation}

For any model $M \models T$, for any $\overline{a}$ from $M$ and any $B$ a set of elements of the same sort as $\overline{y}$ from $M$, let $\tp_\varphi(\overline{a} / B)$ be the $\varphi$-type over $B$ given by:

\begin{equation}
 \tp_\varphi(\overline{a} / B) = \{ \varphi(\overline{x}; \overline{b})^t : \overline{b} \in B, t < 2 \mathrm{\ such\ that\ } M \models \varphi(\overline{a}; \overline{b})^t \}
\end{equation}

The above notions can be defined for sets of formulas $\Gamma(\overline{x}; \overline{y})$ (instead of a single formula) in the obvious way.  Throughout this section, when we mention a $\varphi$-type over $B$, look at $\tp_\varphi(\overline{a} / B)$, or consider the set $S_\varphi(B)$, we want $B$ to be a set of elements of the same sort as $\overline{y}$ (that is, if $\overline{y} = (y_0, ..., y_{n-1})$, then $B$ is a set of $n$-tuples $\overline{b} = (b_0, ..., b_{n-1})$ such that $b_i$ is of the same sort as $y_i$ for all $i < n$).  In Section \ref{Section_proofofmainthm} when we consider $\Delta$-types, we will alter this notation slightly for simplification.  When we consider the set of formulas $\Delta(y; z_0, ..., z_{n-1})$ where all the $z_i$'s are of the same sort and $B$ is a set of elements of that sort, we will abuse notation and say that a $\Delta$-type is \textit{over $B$} when it is actually over $B^n$ and we will write $\tp_\Delta(c / B)$ when we mean $\tp_\Delta(c / B^n)$.

\begin{defn}\label{defdependent}
 We say that a set $B$ of elements of the same sort as $\overline{y}$ is \textbf{$\varphi$-independent} if, for all $s \in {}^B 2$, the set of formulas $\{ \varphi(\overline{x}; \overline{b})^{s(\overline{b})} : \overline{b} \in B \}$ is consistent.  We say that $\varphi$ has \textbf{independence dimension $n < \omega$}, denoted $\ID(\varphi) = n$, if $n$ is maximal such that, for some (equivalently any) model $M \models T$, there exists a set $B$ of elements of the same sort as $\overline{y}$ from $M$ with $|B| = n$ such that $B$ is $\varphi$-independent.  If such an $n$ exists, then we say that $\varphi$ is \textbf{dependent}.  If no such $n$ exists, then we say that $\varphi$ is \textbf{independent}.
\end{defn}

Notice that when $B$ is finite, $B$ is $\varphi$-independent if and only if $|S_\varphi(B)| = 2^{|B|}$.

\begin{defn}\label{phiiso}
 We say that a $\varphi$-type $p(\overline{x})$ is \textbf{$\varphi$-isolated} if there exists a finite $\varphi$-subtype $p_0(\overline{x}) \subseteq p(\overline{x})$ such that $p_0(\overline{x}) \vdash p(\overline{x})$.  We say that a formula $\psi(\overline{x})$ is a \textbf{$\varphi$-formula} if it is of the form $\psi(\overline{x}) = \bigwedge_{i < n} \varphi(\overline{x}; \overline{b}_i)^{s(i)}$ for some $n < \omega$, some elements $\overline{b}_i$ of the same sort as $\overline{y}$, and some $s \in {}^n 2$.
\end{defn}

We see that a $\varphi$-type $p(\overline{x})$ is $\varphi$-isolated if and only if there exists a $\varphi$-formula, $\psi(\overline{x})$ over $\dom(p)$ such that $p(\overline{x})$ is equivalent to $\psi(\overline{x})$.  This $\varphi$-formula is simply the conjunction of the finite $\varphi$-subtype $p_0(\overline{x})$ given in Definition \ref{phiiso}.

For a model $M \models T$ and a set $B$ of elements of same sort as $\overline{y}$ from $M$, consider the language $L_B = L \cup \{ P_B \}$ an expansion of $L$ by adding a single predicate, $P_B(\overline{y})$.  Let $(M; B)$ be the obvious $L_B$-structure.  By ``$(N; B') \succeq (M; B)$'' we mean that $(N; B')$ is an elementary extension of $(M; B)$ in the language $L_B$.

\begin{defn}\label{elmexten}
 Fix $M \models T$ and a set $B$ of elements of the same sort as $\overline{y}$ from $M$.  We say that a $\varphi$-type $p'$ is an \textbf{elementary $\varphi$-extension} of the $\varphi$-type $p \in S_\varphi(B)$ if $p'$ extends $p$ and $\dom(p') \subseteq B'$ for some $(N; B') \succeq (M; B)$.
\end{defn}

Now we are ready to state the main theorem of the paper.  We will give the proof in Section \ref{Section_proofofmainthm} below.

\begin{thm}[The isolated extension theorem]\label{elmextiso}
 For any partitioned formula $\varphi(\overline{x}; \overline{y})$, the following are equivalent:

 \begin{itemize}
  \item [(i)] $\varphi$ is dependent;
  \item [(ii)] For all $\varphi$-types $p$, there exists a $\varphi$-isolated elementary $\varphi$-extension of $p$.
 \end{itemize}

 Moreover, if the above conditions hold, we can choose $p'$ a $\varphi$-isolated elementary $\varphi$-extension of $p \in S_\varphi(B)$ such that $|\dom(p') - B| \le 2 \cdot \ID(\varphi)$.
\end{thm}

We remark on some consequences of the theorem.

\begin{defn}\label{defin}
 Fix a partitioned formula $\varphi(\overline{x}; \overline{y})$, a $\varphi$-type $p(\overline{x})$, and a formula $\psi(\overline{y})$.  We say that $\psi$ \textbf{defines} $p$ if, for all $\overline{b} \in \dom(p)$, $\varphi(\overline{x}; \overline{b}) \in p(\overline{x})$ if and only if $\psi(\overline{b})$ holds.  We say that $\psi$ \textbf{$\varphi$-defines} $p$ if it defines $p$ and it is of the form $\psi(\overline{y}) = \forall \overline{x} ( \gamma(\overline{x}) \rightarrow \varphi(\overline{x}; \overline{y}) )$ for some $\varphi$-formula $\gamma(\overline{x})$.
\end{defn}

Merely requiring that a $\varphi$-type has a defining formula has no content.  Indeed, for any type $p \in S_\varphi(B)$, $p$ is defined by the formula $\varphi(\overline{a}; \overline{y})$ for any realization $\overline{a}$ of $p$.  The strength of having a defining formula is to have one with a controlled domain, preferably over $\dom(p)$.  It is known, for example, that for stable formulas $\varphi$, all $\varphi$-types $p$ have a defining formula over $\dom(p)$ \cite{shelah}, but, when $\dom(p)$ is an arbitrary set, it does not necessarily have a $\varphi$-defining formula over $\dom(p)$.

Notice that if $p$ is $\varphi$-isolated, then $p$ has a $\varphi$-defining formula $\psi$ over $\dom(p)$.  Namely, take the $\varphi$-formula $\gamma$ over $\dom(p)$ such that $p(\overline{x})$ is equivalent to $\gamma(\overline{x})$ and let $\psi(\overline{y}) = \forall \overline{x} ( \gamma(\overline{x}) \rightarrow \varphi(\overline{x}; \overline{y}) )$.  It is clear that if $\psi$ $\varphi$-defines $p$, then $\psi$ defines $p$, but the converse does not necessarily hold.  We immediately get the following corollary to Theorem \ref{elmextiso}.

\begin{cor}[Elementary $\varphi$-definability of types]\label{elmdeftypes}
 If $M \models T$, $\overline{y}$ is a list of variables, and $B$ is a set of elements of same sort as $\overline{y}$ from $M$, then there exists an elementary extension $(N; B') \succeq (M; B)$ such that, for all dependent formulas $\varphi(\overline{x}; \overline{y})$, for all $p(\overline{x}) \in S_\varphi(B)$, there exists $\psi(\overline{y})$ over $B'$ such that $\psi$ $\varphi$-defines $p$.
\end{cor}

\begin{proof}
 Fix $M \models T$ and $B$ from $M$ of the appropriate sort, and fix $(N; B') \succeq (M; B)$ sufficiently saturated.  Then, by Theorem \ref{elmextiso}, there exists $p'$ a $\varphi$-isolated elementary $\varphi$-extension of $p$ (with $\dom(p') \subseteq B'$).  Since $p'$ is $\varphi$-isolated, there exists $\psi$ (over $\dom(p') \subseteq B'$) that $\varphi$-defines $p'$.  Since $p \subseteq p'$, $\psi$ $\varphi$-defines $p$.
\end{proof}

Notice that Corollary \ref{elmdeftypes} is, on the one hand, stronger than standard definability of types for stable formulas, and, on the other hand, weaker.  We get that, for dependent formulas $\varphi$, $\varphi$-types are not only definable, but $\varphi$-definable.  However, the formula doing the defining is not over $\dom(p)$, but over $B'$ for some $(N; B') \succeq (M; \dom(p))$.

As in the stable case, this $\varphi$-definability of types leads to a notion of stable embeddability.

\begin{cor}[Elementary stable embeddability]\label{elmstabembed}
 If $M \models T$ for a dependent theory $T$, $\overline{y}$ is a list of variables, and $B$ is a set of elements of same sort as $\overline{y}$ from $M$, then there exists an elementary extension $(N; B') \succeq (M; B)$ such that, for all formulas $\varphi(\overline{y})$ over any elementary supermodel of $M$, there exists a formula $\psi(\overline{y})$ over $B'$ such that $\varphi(B) = \psi(B)$.  Moreover, $\psi(\overline{y}) = \forall \overline{x} ( \gamma(\overline{x}) \rightarrow \varphi(\overline{x}; \overline{y}) )$ for some $\varphi$-formula $\gamma$.
\end{cor}

\begin{proof}
 Fix $(N; B') \succeq (M; B)$ sufficiently saturated as above.  For any fixed formula $\varphi(\overline{y})$, say $\varphi$ is over $N' \succeq M$, let $\varphi(\overline{y}) = \varphi_0(\overline{a}; \overline{y})$ for $\varphi_0(\overline{x}; \overline{y})$ over $\emptyset$ and $\overline{a}$ from $N'$, and let $p(\overline{x}) = \tp_{\varphi_0}(\overline{a} / B)$.  As $\varphi_0$ is dependent, by Corollary \ref{elmdeftypes}, there exists $\psi(\overline{y})$ over $B'$ that $\varphi_0$-defines $p$.  Then, by definition, $\varphi(B) = \psi(B)$.
\end{proof}

\section{The Proof of the Isolated Extension Theorem}\label{Section_proofofmainthm}

To aid notation, assume that the length of $\overline{x}$ and the length of $\overline{y}$ is $1$.  Other than having more complicated notation, the general case is identical.

First, to show (ii) implies (i), we will exhibit the contrapositive.  Assume then that $\varphi(x;y)$ is independent.  By compactness, there exists a model $M$ with an infinite $\varphi$-independent set $B$.  Let $(N; B') \succeq (M; B)$.  By elementarity, it follows that all finite subsets of $B'$ are $\varphi$-independent.  Let $p'$ be any extension of $p$ to a $\varphi$-type such that $\dom(p') \subseteq B'$.  Fix any finite subtype $p_0(x) \subseteq p'(x)$.  Now, for any finite $\varphi$-type $p_1(x)$ with $p_0(x) \subsetneq p_1(x) \subseteq p'(x)$, since $\dom(p_1)$ is $\varphi$-independent, we cannot have that $p_0(x) \vdash p_1(x)$.  Thus, $p_0(x) \not\vdash p'(x)$.  This shows that no elementary $\varphi$-extension of $p$ is $\varphi$-isolated.  Therefore, (ii) implies (i).

To show (i) implies (ii), we will first show that the following proposition holds:

\begin{prop}\label{AlmostDef}
 For any dependent formula $\varphi(x;y)$ in a theory $T$, for any model $M \models T$, for any partial type $\Theta(y)$ over $\emptyset$, and for any $B \subseteq \Theta(M)$, there exists $N \succeq M$ and $C \subseteq \Theta(N)$ with $|C| \le 2 \cdot \ID(\varphi)$ and an extension $p'(x) \in S_\varphi(B \cup C)$ of $p(x)$ that is $\varphi$-isolated.
\end{prop}

Fix $\varphi(x;y)$ a dependent formula in a theory $T$ and $\Theta(y)$ any partial type over $\emptyset$.  Let $n = \ID(\varphi)$, the independence dimension of $\varphi(x;y)$.  Fix $M \models T$, $N \succeq M$ sufficiently saturated, $B \subseteq \Theta(M)$, and $p(x) \in S_\varphi(B)$.  If $B$ is finite, $p$ is already isolated, so assume that $B$ is infinite.  Define a set of formulas $\Delta(y;z_0, ..., z_{n-1})$ as follows:

\begin{equation}
 \Delta(y;z_0, ..., z_{n-1}) = \left\{ \exists x \left( \varphi(x; y)^t \wedge \bigwedge\limits_{i < n} \varphi(x; z_i)^{s(i)} \right) : t < 2, s \in {}^n 2 \right\}
\end{equation}

We will now define the notion of a good configuration.  This will end up allowing us to build up the external $C$ in at most $\ID(\varphi)$ steps (adding two elements at a time).

\begin{defn}\label{gconfig}
 A \textbf{good configuration} of $p$ of size $K$ is a sequence $C = \{ c_{i,t} : i < K, t < 2 \}$ such that the following conditions hold:

 \begin{itemize}
  \item [(i)] $c_{i,t} \models \Theta(y)$ for all $i < K$, $t < 2$;
  \item [(ii)] $p(x) \cup \{ \varphi(x; c_{j,t})^t : j < K, t < 2 \}$ is consistent; and
  \item [(iii)] For all $s \in {}^K 2$, all $j < K$, $c_{j,0}$ and $c_{j,1}$ have the same $\Delta$-type over $B \cup \{ c_{i,s(i)} : i \neq j \}$.
 \end{itemize}

 If $C$ is a good configuration of $p$ of size $K$, then let $p_C(x) = p(x) \cup \{ \varphi(x; c_{j,t})^t : j < K, t < 2 \}$.
\end{defn}

The first thing to note is that these good configurations are used to extend the type $p$ in a very specific way.  These could, \textit{a priori}, be arbitrarily large.  However, the fact that $\varphi$ is dependent forces good configurations to be of bounded size.

\begin{lem}\label{gconfigbound}
 If $C = \{ c_{i,t} : i < K, t < 2 \}$, is a good configuration of $p$ of size $K$, then $K \le n = \ID(\varphi)$.
\end{lem}

\begin{proof}
 Suppose not, i.e. $K > n$.  Now, for each $s \in {}^{n+1} 2$, notice that

 \begin{equation}\label{gcexiststatement}
  \models \exists x \bigwedge_{i < n+1} \varphi(x; c_{i, s(i)})^{s(i)}
 \end{equation}

 because $\{ \varphi(x; c_{i, s(i)})^{s(i)} : i < n + 1 \}$ is a consistent type.  Now, notice that, for any $j \le n$,

 \begin{multline}\label{keygcjump}
  \exists x \left( \bigwedge_{i < j} \varphi(x; c_{i, 0})^{s(i)} \wedge \varphi(x; c_{j, s(j)})^{s(j)} \wedge \bigwedge_{j < i < n+1} \varphi(x; c_{i, s(i)})^{s(i)} \right) \Longrightarrow \\
  \exists x \left( \bigwedge_{i < j} \varphi(x; c_{i, 0})^{s(i)} \wedge \varphi(x; c_{j, 0})^{s(j)} \wedge \bigwedge_{j < i < n+1} \varphi(x; c_{i, s(i)})^{s(i)} \right)
 \end{multline}

 because $c_{j,0}$ and $c_{j,1}$ have the same $\Delta$-type over $\{ c_{i,0} : i < j \} \cup \{ c_{i, s(i)} : j < i < n + 1 \}$.  Starting with (\ref{gcexiststatement}), then using (\ref{keygcjump}) and induction, we get that:

 \begin{equation}
  \models \exists x \bigwedge_{i < n+1} \varphi(x; c_{i, 0})^{s(i)}
 \end{equation}

 But this holds for any $s \in {}^{n+1} 2$.  This contradicts the fact that $n = \ID(\varphi)$.
\end{proof}

Now that we have good configurations, we need a sufficient condition for taking a good configuration and building a larger one out of it.  Clearly any new $d_0$ and $d_1$ we would like to add on must realize $\Theta$ and must be so that $\neg \varphi(x; d_0) \wedge \varphi(x; d_1)$ is consistent with $p_C(x)$.  However, the third condition for a good configuration is a bit tricky.  Not only do $d_0$ and $d_1$ have to have the same $\Delta$-type over $B \cup \{ c_{i,s(i)} : i < K \}$, but also each $c_{j,0}$ and $c_{j,1}$ have to have the same $\Delta$-type over $B \cup \{ c_{i,s(i)} : i \neq j \} \cup \{ d_t \}$.  We now give a sufficient condition for being able to add on to good configurations.

\begin{lem}\label{addtogconfig}
 If $C = \{ c_{i,t} : i < K, t < 2 \}$ is a good configuration of $p$, and there exists $d_0$, $d_1$ such that:

 \begin{itemize}
  \item [(i)] $d_0, d_1 \models \Theta(y)$;
  \item [(ii)] $p_C(x) \cup \{ \varphi(x; d_t)^t : t < 2 \}$ is consistent;
  \item [(iii)] $\tp_\Delta(d_0 / B \cup C) = \tp_\Delta(d_1 / B \cup C)$; and
  \item [(iv)] $\tp_\Delta(d_0 / B \cup C)$ is finitely satisfiable in $B$.
 \end{itemize}

 Then, $C \cup \{ d_0, d_1 \}$ is a good configuration of $p$ (of size $K + 1$).
\end{lem}

\begin{proof}
 Clearly all conditions for $C \cup \{ d_0, d_1 \}$ to be a good configuration of $p$ are met except perhaps the condition that $c_{j,0}$ and $c_{j,1}$ have the same $\Delta$-type over $B \cup \{ c_{i,s(i)} : i \neq j \} \cup \{ d_t \}$ for all $s \in {}^K 2$, $t < 2$.  So suppose this fails, and fix the $s \in {}^K 2$ and $t < 2$ where this fails.

 Then there exists $\delta$ either an element of $\Delta$ or the negation of an element of $\Delta$ such that $N \models \delta(c_{j,0}, \overline{e}) \wedge \neg \delta(c_{j,1}, \overline{e})$ for some $\overline{e}$ from $B \cup \{ c_{i,s(i)} : i \neq j \} \cup \{ d_t \}$.  Since $c_{j,0}$ and $c_{j,1}$ have the same $\Delta$-type over $B \cup \{ c_{i,s(i)} : i \neq j \}$, we must have that $\overline{e} = d_t \concat \overline{e}'$ for some $\overline{e}'$ from $B \cup \{ c_{i,s(i)} : i \neq j \}$.  Therefore, we get that:

 \begin{equation}
  N \models \delta(c_{j,0}, d_t, \overline{e}') \wedge \neg \delta(c_{j,1}, d_t, \overline{e}')
 \end{equation}

 By condition (iv) of the hypothesis, there exists $b \in B$ such that:

 \begin{equation}
  N \models \delta(c_{j,0}, b, \overline{e}') \wedge \neg \delta(c_{j,1}, b, \overline{e}')
 \end{equation}

 But, as $b \concat \overline{e}'$ is from $B \cup \{ c_{i,s(i)} : i \neq j \}$, this contradicts the fact that $c_{j,0}$ and $c_{j,1}$ have the same $\Delta$-type over $B \cup \{ c_{i,s(i)} : i \neq j \}$.
\end{proof}

Fix $C$ a maximal good configuration of $p$, so $p_C(x)$ is a $\varphi$-type over $B \cup C$.  Let $s(x)$ be any extension of $p_C(x)$ to a complete type over $B \cup C$.  Define $r_s(y)$ as follows:

\begin{equation}\label{defofrs}
 r_s(y) = \{ \exists x (\varphi(x; y)^t \wedge \psi(x)) : \psi \in s, t < 2 \} \cup \Theta(y)
\end{equation}

\begin{lem}\label{rnotfinsat}
 $r_s$ is not finitely satisfied in $B$.
\end{lem}

\begin{proof}
 Suppose, by means of contradiction, that $r_s$ is finitely satisfied in $B$.  Let $\mathcal{D}$ be an ultrafilter on $B$ such that for all $\delta(y) \in r_s(y)$, $\delta(B) \in \mathcal{D}$ (this exists by finite satisfiability of $r_s$ in $B$).  Let $q(y) = \Avtp(\mathcal{D}, B \cup C)$, the average type of $\mathcal{D}$ over $B \cup C$.  That is, for any formula $\delta(y)$ over $B \cup C$, $\delta(y) \in q(y)$ if and only if $\delta(B) \in \mathcal{D}$.  Then $q \in S(B \cup C)$, $q$ extends $r_s$, and $q$ is finitely satisfied in $B$.  Let $q' = q \upharpoonright_\Delta$.

 Now notice that $\{ \exists x (\varphi(x; y)^t \wedge \psi(x)) \} \cup q(y)$ is consistent for each $\psi \in s$ and each $t < 2$.  Since $s$ is closed under conjunction, by compactness we get that $s(x) \cup \{ \varphi(x; y)^t \} \cup q(y)$ is consistent for each $t < 2$.  Therefore, $s(x) \cup \{ \varphi(x; y)^t \} \cup q'(y) \cup \{ \theta(y) \}$ is consistent for each $t < 2$ and each $\theta(y)$ a finite conjunction of formulas from $\Theta(y)$ (as $q'(y) \cup \Theta(y) \subseteq q(y)$).  This means that $s(x) \cup \{ \exists y ( \varphi(x; y)^t \wedge \theta(y) \wedge \psi(y) ) \}$ is consistent for each $\psi(y)$ a finite conjunction of formulas from $q'(y)$ and each $\theta(y)$ a finite conjunction of formulas from $\Theta(y)$.  But, since $s$ is a complete type in the $x$ variable, $s$ decides all formulas of the form $\exists y ( \varphi(x; y)^t \wedge \theta(y) \wedge \psi(y) )$.  Therefore, we get that:

 \begin{equation}
  \exists y ( \varphi(x; y)^t \wedge \theta(y) \wedge \psi(y) ) \in s(x)
 \end{equation}

 Choose $\psi_t(x)$ a finite conjunction of formulas from $q'(y)$ and $\theta_t(y)$ a finite conjunction of formulas from $\Theta(y)$ for both $t < 2$.  Then $\exists y_t ( \varphi(x; y_t)^t \wedge \theta_t(y_t) \wedge \psi_t(y_t) ) \in s(x)$ for both $t < 2$.  Therefore, we get that:

 \begin{equation}
  s(x) \cup \{ \exists y_0 ( \neg \varphi(x; y_0) \wedge \theta_0(y_0) \wedge \psi_0(y_0) ) \} \cup \{ \exists y_1 ( \varphi(x; y_1) \wedge \theta_1(y_1) \wedge \psi_1(y_1) ) \}
 \end{equation}

 is consistent.  Now, by compactness,

 \begin{equation}
  u(x,y_0,y_1) = s(x) \cup \{ \neg \varphi(x; y_0) \wedge \varphi(x; y_1) \} \cup q'(y_0) \cup q'(y_1) \cup \Theta(y_0) \cup \Theta(y_1)
 \end{equation}

 is consistent.  So, taking any realization $(a, d_0, d_1)$ of $u(x,y_0,y_1)$ from $N$, we see that $d_0, d_1 \models \Theta(y)$, $d_0, d_1 \models q'(y)$, and $p_C(x) \cup \{ \varphi(x; d_t)^t : t < 2 \}$ is consistent.  So conditions (i), (ii), and (iii) of Lemma \ref{addtogconfig} are met.  However, since $q$ is finitely satisfied in $B$, $q'$ is finitely satisfied in $B$.  Therefore, condition (iv) of Lemma \ref{addtogconfig} is met, so $C \cup \{ d_0, d_1 \}$ is a good configuration of $p$.  This contradicts the maximality of $C$.
\end{proof}

We will now show how the non-finite-satisfiability of $r_s$ in $B$ leads to a formula definition of $p_C(x)$.

\begin{lem}\label{defnblty}
 For any $C$ a maximal good configuration of $p$ and any $s(x) \in S(B \cup C)$ an extension of $p_C(x)$, there exists a formula $\gamma(x) \in s(x)$ such that $\gamma(x) \vdash p_C(x)$.
\end{lem}

\begin{proof}
 Consider $r_s$ as given above.  Then, since $r_s$ is not finitely satisfiable in $B$, there exists $m < \omega$ and $\psi_\ell(x) \in s(x)$ for each $\ell < m$ such that, for all $b \in B$, $N \models \neg \exists x (\varphi(x; b)^t \wedge \psi_\ell(x))$ for some $\ell < m$ and some $t < 2$ (notice here that $b \models \Theta(y)$ for all $b \in B$, so that the formulas in $\Theta(y) \subseteq r_s(y)$ are always realized in $B$).  Let $\gamma(x)$ be defined as follows:

 \begin{equation}
  \gamma(x) = \bigwedge_{\ell < m} \psi_\ell(x) \wedge \bigwedge_{i < K, u < 2} \varphi(x; c_{i,u})^u .
 \end{equation}

Since $s$ is closed under conjunction, $s$ extends $p_C$, and $\psi_\ell(x) \in s(x)$, we get that $\gamma(x) \in s(x)$.  To prove that $\gamma(x) \vdash p_C(x)$, notice that, for all $b \in B$, there exists $t < 2$ such that $N \models \forall x (\bigwedge_{\ell < m} \psi_\ell(x) \rightarrow \varphi(x;b)^t)$.  Therefore, $s(x) \vdash \gamma(x) \vdash \varphi(x; b)^t$, hence $\varphi(x; b)^t \in s(x)$.  But $s$ extends $p_C$, so we get that $\varphi(x; b)^t \in p_C(x)$.  Similarly, $\gamma(x) \vdash \varphi(x; c_{i, u})^u$ for all $i < K$ and $u < 2$.  Therefore, $\gamma(x) \vdash p_C(x)$.
\end{proof}

Now that we have a formula definition for $p_C(x)$ for each $s \in S(B \cup C)$, we will see that a single formula is equivalent to $p_C(x)$ using compactness.  After that, we will show that this means a finite $\varphi$-subtype of $p_C(x)$ is equivalent to the whole of $p_C(x)$.

\begin{lem}\label{determinetype}
 If $C = \{ c_{i,t} : i < K, t < 2 \}$ is a maximal good configuration of $p$, then there exists a formula $\psi(x)$ over $B \cup C$ such that $\psi(x)$ is equivalent to $p_C(x)$.
\end{lem}

\begin{proof}
 For each such $s(x) \in S(B \cup C)$ extending $p_C(x)$, define $\gamma_s(x)$ to be the formula such that $\gamma_s(x) \in s(x)$ and $\gamma_s(x) \vdash p_C(x)$ as given in Lemma \ref{defnblty}.

 Consider the following partial type over $B \cup C$:

 \begin{equation}
  \Sigma(x) = \{ \neg \gamma_s(x) : s \in S(B \cup C) \mathrm{\ and\ } s(x) \supseteq p_C(x) \} \cup p_C(x)
 \end{equation}

 Now $\Sigma(x)$ is inconsistent, since otherwise we would have $a \models p_C(x)$ yet $a \not\models \gamma_s(x)$ for any $s(x)$ extending $p_C(x)$.  In particular, $a \not\models \gamma_{s_0}(x)$ for $s_0 = \tp(a/ B \cup C)$.  This contradicts the fact that $s_0(x) \vdash \gamma_{s_0}(x)$.  Therefore, by compactness, there exists some finite set $S_0 \subseteq S(B \cup C)$ of types extending $p_C$ so that $\Sigma_0(x) = \{ \neg \gamma_s(x) : s \in S_0 \} \cup p_C(x)$ is inconsistent.  Let $\psi(x) = \bigvee_{s \in S_0} \gamma_s(x)$.

 Certainly $\psi(x) \vdash p_C(x)$ as $\gamma_s(x) \vdash p_C(x)$ for all $s \in S_0$.  Conversely, if $a \models p_C(x)$, then $a \not\models \{ \neg \gamma_s(x) : s \in S_0 \}$ (by the inconsistency of $\Sigma_0(x)$).  Therefore, $a \models \psi(x)$.  Hence, $p_C(x) \vdash \psi(x)$, as desired.
\end{proof}

\begin{lem}\label{finitecoll}
 If $C = \{ c_{i,t} : i < K, t < 2 \}$ is a maximal good configuration of $p$, then there exists a finite $\varphi$-subtype $p_0(x) \subseteq p_C(x)$ so that $p_0(x) \vdash p_C(x)$.
\end{lem}

\begin{proof}
 First let $\psi(x)$ be a formula over $B \cup C$ that is equivalent to $p_C(x)$, given by Lemma \ref{determinetype}.  Then consider $\{ \neg \psi(x) \} \cup p_C(x)$, a partial type over $B \cup C$.  This is clearly inconsistent.  Therefore, there exists a finite subset $p_0(x) \subseteq p_C(x)$ such that $\{ \neg \psi(x) \} \cup p_0(x)$ is inconsistent.  That is, $p_0(x) \vdash \psi(x)$ and, therefore, we get that $p_0(x) \vdash p_C(x)$.
\end{proof}

We are now ready to prove Proposition \ref{AlmostDef}.

\begin{proof}[Proof of Proposition \ref{AlmostDef}]
 Take $C = \{ c_{i,t} : i < K, t < 2 \}$ any maximal good configuration of $p$.  By definition, $C \subseteq \Theta(N)$.  By Lemma \ref{gconfigbound}, $K \le n$, hence $|C| \le 2 \cdot n$.  Let $p'(x) = p_C(x) = p(x) \cup \{ \varphi(x; c_{i,t})^t : i < K, t < 2 \}$.  By Lemma \ref{finitecoll}, there exists a finite $p_0(x) \subseteq p'(x)$ so that $p_0(x) \vdash p'(x)$.  Therefore, $p'(x)$ is $\varphi$-isolated.
\end{proof}

From here we can conclude that (i) implies (ii) holds for Theorem \ref{elmextiso}.

Let $\varphi(x;y)$ be dependent, fix $M \models T$, $B \subseteq M$ of elements of the same sort as $y$, and any $\varphi$-type $p \in S_\varphi(B)$.  Let $\Theta(y) = \{ P_B(y) \}$ in the language $L_B$ (notice that $\varphi(x;y)$ is still dependent with the same independence dimension in the theory $Th_{L_B}(M; B)$).  Therefore, by Proposition \ref{AlmostDef}, there exists $(N; B') \succeq (M; B)$ and $C \subseteq \Theta((N; B')) = B'$ with $|C| \le 2 \cdot \ID(\varphi)$ and a type $p'(x) \in S_\varphi(B \cup C)$ extending $p(x)$ such that $p'(x)$ is $\varphi$-isolated.  Notice then that $p'$ is an elementary $\varphi$-extension of $p$ that is $\varphi$-isolated, so condition (ii) holds.  Moreover, we get that $|\dom(p') - B| = |C| \le 2 \cdot \ID(\varphi)$, as desired. \hfill $\square$

\begin{rem}
 Finally, we remark that this $C$, hence $p'$, depends only on a type over $B$ with enough information to guarantee that $C$ is a good configuration of $p$ of maximal size.  For example, if we take $\overline{c} = (c_{i,t} : i < K, t < 2)$ for $C = \{ c_{i,t} : i < K, t < 2 \}$ a good configuration of $p$ of maximal size, and let

 \begin{itemize}
  \item [(i)] $q'(\overline{y}) = \{ P_B(y_{i,t}) : i < K, t < 2 \}$, 
  \item [(ii)] $q''(\overline{y}) = \left\{ \exists x \left( \psi(x) \wedge \bigwedge\limits_{i < K, t < 2} \varphi(x; c_{i,t})^t \right) : \psi(x) \mathrm{\ a\ finite\ conjunction\ from\ } p(x) \right\}$,
  \item [(iii)] $q'''(\overline{y}) = \tp_\Delta(\overline{c} / B)$, and
  \item [(iv)] $q(\overline{y}) = q'(\overline{y}) \cup q''(\overline{y}) \cup q'''(\overline{y})$,
 \end{itemize}

 then, for any $\overline{c}' \models q(\overline{y})$, the type $p_{\overline{c}'}(x) = p(x) \cup \{ \varphi(x; c'_{i, t})^t : i < K, t < 2 \}$ is $\varphi$-isolated (and an elementary $\varphi$-extension of $p$).  Notice here that $q(\overline{y}) \subseteq \tp(\overline{c} / B)$, the complete type of $\overline{c}$ over $B$.  Therefore, so long as we choose $(N; B') \succeq (M; B)$ so that $(N; B')$ is $|B|^+$-saturated, $q(\overline{y})$ is realized in $N$.  This allows us to pick $(N; B') \succeq (M; B)$ uniformly so that all $\varphi$-types over $B$ have extensions to $\varphi$-isolated $\varphi$-types with domain contained in $B'$.
\end{rem}

\section{$\varphi$-Isolated Elementary $\varphi$-Extensions for Stable $\varphi$}\label{Section_stablesetting}

Since stable formulas are, in particular, dependent, all stable formulas have the property of Theorem \ref{elmextiso} (ii).  But what is the $\varphi$-isolated elementary $\varphi$-extension $p'(x)$ of a given $\varphi$-type $p(x)$?  In the interesting case when $p(x)$ is not already $\varphi$-isolated, $p'(x)$ is a \textit{forking} extension of $p(x)$.  This follows from the Open Mapping Theorem (i.e. the fact that the restriction map from non-forking $\varphi$-extensions of $S_\varphi(A)$ to $S_\varphi(A)$ is open) as, if $p$ has a non-forking $\varphi$-isolated extension, then it is already $\varphi$-isolated.

On the issue of uniformity, the results of Theorem \ref{elmextiso} differ strongly from the standard definability of $\varphi$-types in the stable case.  In the case where $\varphi$ is stable, we can use a compactness argument to get a uniform definition of $\varphi$-types.  Note, however, that this uniform definition is not necessarily a $\varphi$-definition.  One cannot, in general, get a uniform $\varphi$-definition of all $\varphi$-types, even in the case where $\varphi$ is stable.

As an example, let $T$ be the theory, in the language $L = \{ E \}$ with a single binary relation $E$, stating that $E$ is an equivalence relation with infinitely many $E$-equivalence classes all of infinite size.  This theory is certainly stable, and even $\aleph_0$-stable.  Fix $M \models T$ and let $B \subset M$ be a set containing one element from one class, two from another, three from a third class, and so on.  Finally, let $\varphi(x; y,z,w)$ be the formula given by:

\begin{equation}
 \varphi(x; y, z, w) = [ (z = w \rightarrow x = y) \wedge (z \neq w \rightarrow E(x,y)) ]
\end{equation}

(so $\varphi$ encodes the two formulas ``$x=y$'' and $E(x,y)$ into a single formula).  Now let $n \in \omega$ be arbitrary and let $a \in M - B$ be in the $E$-equivalence class with exactly $n$ elements of $B$ in it; call this class $[a]_E$.  Finally, let $p_n(x) = \tp_\varphi(a / B)$.  Now, for any $(N; B') \succeq (M; B)$, notice that the $E$-equivalence class with exactly $n$ elements from $B$ still has exactly $n$ elements from $B'$, so $[a]_E \cap B' = [a]_E \cap B$.  However, this shows that any $\varphi$-extension of $p_n$ to some $p'$ with $\dom(p') \subseteq B'$ is $\varphi$-isolated only by a finite subtype whose domain contains $[a]_E \cap B$ (this is because we need the full set $[a]_E \cap B$ to say that $x \neq b$ for each $b \in ([a]_E \cap B)$ yet $E(x,b)$ for some (all) $b \in ([a]_E \cap B)$).  As $|[a]_E \cap B| = n$ and $n < \omega$ was arbitrary, we see from this example that there is no \textit{uniform} bound on the size of the $\varphi$-isolating $\varphi$-subtype of the elementary $\varphi$-extension given by Theorem \ref{elmextiso}, even in the stable case.

\end{document}